\documentclass[11pt, notitlepage]{article}
\usepackage{amssymb,amsmath,comment}
\catcode`\@=11 \@addtoreset{equation}{section}
\def\thesection{\arabic{section}}

\def\theequation{\thesection.\arabic{equation}}
\catcode`\@=12
\usepackage{colortbl}%
\usepackage[mathscr]{eucal}
\usepackage{epsf}
\usepackage{a4wide}

\newcommand{\e}{\epsilon}

\newcommand{\Om} {\Omega}

\newcommand{\De} {\Delta}
\newcommand{\la} {\lambda}

\newcommand{\noi} {\noindent}

\newcommand{\mb} {\mathbb}
\newcommand{\mc} {\mathcal}

\usepackage[all]{xy}
\catcode`\@=11
\def\theequation{\@arabic{\c@section}.\@arabic{\c@equation}}
\catcode`\@=12

\newcommand{\QED}{\rule{2mm}{2mm}}

\newtheorem{Theorem}{Theorem}[section]
\newtheorem{Lemma}[Theorem]{Lemma}
\newtheorem{Proposition}[Theorem]{Proposition}
\newtheorem{Corollary}[Theorem]{Corollary}
\newtheorem{Remark}[Theorem]{Remark}
\newtheorem{Definition}[Theorem]{Definition}

\begin{document}

{\vspace{0.01in}}

\title
{ \textsc{On doubly nonlocal $p$-fractional coupled elliptic system} }

\author{ T. Mukherjee\footnote{Department of Mathematics, Indian Institute of Technology Delhi,
Hauz Khaz, New Delhi-110016, India.
 e-mail: tulimukh@gmail.com}~ and ~K. Sreenadh\footnote{Department of Mathematics, Indian Institute of Technology Delhi,
Hauz Khaz, New Delhi-110016, India.
 e-mail: sreenadh@gmail.com} }

\date{}

\maketitle

\begin{abstract}

\noi We study the following nonlinear system with perturbations involving p-fractional Laplacian
\begin{equation*}
(P)\left\{
\begin{split}
(-\De)^s_p u+ a_1(x)u|u|^{p-2} &= \alpha(|x|^{-\mu}*|u|^q)|u|^{q-2}u+ \beta (|x|^{-\mu}*|v|^q)|u|^{q-2}u+ f_1(x)\; \text{in}\; \mb R^n,\\
(-\De)^s_p v+ a_2(x)v|v|^{p-2} &= \gamma(|x|^{-\mu}*|v|^q)|v|^{q-2}v+ \beta (|x|^{-\mu}*|u|^q)|v|^{q-2}v+ f_2(x)\; \text{in}\; \mb R^n,
\end{split}
\right.
\end{equation*}
where $n>sp$, $0<s<1$, $p\geq2$, $\mu \in (0,n)$, $\frac{p}{2}\left( 2-\frac{\mu}{n}\right) < q <\frac{p^*_s}{2}\left( 2-\frac{\mu}{n}\right)$, $\alpha,\beta,\gamma >0$, $0< a_i \in C^1(\mb R^n, \mb R)$, $i=1,2$ and $f_1,f_2: \mb R^n \to \mb R$ are perturbations. We show existence of atleast two nontrivial solutions for $(P)$ using Nehari manifold and minimax methods.

\medskip

\noi \textbf{Key words:} p-fractional Laplacian, choquard equation, Nehari manifold.

\medskip

\noi \textit{2010 Mathematics Subject Classification:} 35R11, 35R09, 35A15.

\end{abstract}

\section{Introduction and Main results}
In this article, we consider the following nonlinear system with perturbations involving p-fractional Laplacian
\begin{equation*}
(P)\left\{
\begin{split}
(-\De)^s_p u+ a_1(x)u|u|^{p-2} &= \alpha(|x|^{-\mu}*|u|^q)|u|^{q-2}u+ \beta (|x|^{-\mu}*|v|^q)|u|^{q-2}u+ f_1(x)\; \text{in}\; \mb R^n,\\
(-\De)^s_p v+ a_2(x)v|v|^{p-2} &= \gamma(|x|^{-\mu}*|v|^q)|v|^{q-2}v+ \beta (|x|^{-\mu}*|u|^q)|v|^{q-2}v+ f_2(x)\; \text{in}\; \mb R^n,
\end{split}
\right.
\end{equation*}
where $p\geq 2, s\in (0,1), n>sp$, $\mu \in (0,n)$, $\frac{p}{2}\left( 2-\frac{\mu}{n}\right) < q <\frac{p^*_s}{2}\left( 2-\frac{\mu}{n}\right)$, $\alpha,\beta,\gamma >0$, $0< a_i \in C^1(\mb R^n, \mb R)$, $i=1,2$ and $f_1,f_2: \mb R^n \to \mb R$ are perturbations. Here
 $p^*_s = \frac{np}{n-sp}$
  is the critical exponent associated with the embedding of the fractional Sobolev space $W^{s,p}(\mb R^n).$
 The $p$-fractional Laplace operator is defined on smooth functions as
\[(-\De)^s_pu(x)= 2 \lim_{\e \searrow 0} \int_{|x|>\e} \frac{|u(x)-u(y)|^{p-2}(u(x)-u(y))}{|x-y|^{n+sp}}~dy\]
which is nonlinear and nonlocal in nature. This definition matches to linear fractional Laplacian operator $(-\De)^s$, upto a normalizing constant depending on $n$ and $s$, when $p=2$. $(-\De)^s_p$ is degenerate when $p>2$ and singular when $1<p<2$. For more details and motivations and  the function spaces $W^{s,p}(\Om)$,  we refer to \cite{cafe, hitch}.  Researchers are paying a lot of attention to the study of fractional and non-local operators of elliptic type due to concrete real world applications in finance, thin obstacle problem, optimization, quasi-geostrophic flow etc. The eigenvalue problem involving $p$-fractional Laplace equations has been extensively studied in \cite{brasco1,brasco2,sq1,linq}. The Brezis Nirenberg type problem involving $p$-fractional Laplacian has been studied in \cite{sq2} whereas existence has been investigated via Morse theory in \cite{iano}. Problems involving $p$-fractional Laplacian has been studied in \cite{ss2,ss1} using Nehari manifold.
A vast amount of literature can be found for the case $p=2$, i.e., fractional Laplacian $(-\De)^s$, which are contributed in recent years. Some of them includes work of Servadei and Valdinoci in \cite{serv1,serv2,serv4} on bounded domains.\\

\noi The study of fractional Schr$\ddot{\text{o}}$dinger equations has attracted the attention of many researchers now a days. Fr\"{o}lich et al. studied nonlinear Hartree equations in \cite{FTY1,FTY2}.  In the nonlocal case, using variational methods and the Ljusternik –Schnirelmann category theory, L\"{u} and Xu \cite{LuXu} proved existence and multiplicity for the equation
\[\e^{2s}(-\De)^su + V(x)u = \e^{-\alpha}(W_\alpha(x)* |u|^{p})|u|^{p-2}u \; \text{in}\; \mb R^n,\]
where $\e> 0$ is a parameter, $0 <s< 1, N > 2s, V(x)$ is a continuous potential, and $W_\alpha(x)$ is the Riesz potential. Wu  in \cite{Wu} proved the existence of standing waves by studying the related constrained minimization problems via the concentration-compactness principle for the following  nonlinear fractional Schr\"{o}dinger equations with Hartree type nonlinearity
\[i \psi_t + (-\De)^\alpha\psi- (|\cdot|^{-\gamma}*|\psi|^2)\psi=0,\]
where $ 0<\alpha<1, 0<\gamma<2\alpha$ and $\psi(x,t)$ is a complex-valued function on $\mb R^d\times \mb R, d\geq1$. Some recent works on Sch\"{o}dinger equations with fractional Laplacian equation includes \cite{ASS,Ge,TS-cho,SGY} with no attempt to provide a complete list.  Existence of solutions for the equation of the type
\[ -\De u + w(x)u = (I_\alpha * |u|^p)|u|^{p-2}u \; \text{in} \; \mb R^n,\]
where $w(x)$ is appropriate function, $I_\alpha$ is Reisz potential and $p>1$ is chosen appropriately,
have been studied in \cite{ANY, clsa, ghsc, mosc, wang}.  Very recently, Ghimenti, Moroz and Schaftingen \cite{GMS}  proved the existence of least action nodal solution for the above problem taking $w\equiv 1$ and $p=2$. Alves, Figueiredo and Yang \cite{AFY} proved existence of a nontrivial solution via penalization method for the following Choquard equation
\begin{equation*}
-\De u+V(x)u= (|x|^{-\mu}*F(u))f(u)\; \text{in}\; \mb R^n,
\end{equation*}
 where $0 <\mu < N,\; N = 3, \;V$ is a continuous real function and $F$ is the primitive function of $f$. Alves and Yang also studied quasilinear Choquard equation in \cite{ay1,ay2,ay3}. For more study, we also refer \cite{moroz1,moroz2,moroz3} to the readers.\\

\noi System of elliptic equations involving fractional Laplacian and homogeneous nonlinearity has been studied in \cite{GPS,Sqsn} and $p$-fractional elliptic systems has been studied in \cite{CD,CS}  using Nehari manifold techniques. Very recently, Guo et al. \cite{guo} studied  a nonlocal system involving fractional Sobolev critical exponent and fractional Laplacian. There are not many results on elliptic systems with non-homogeneous nonlinearities in the literature.  We also cite \cite{choi,faria,bisci,ww} as some very recent works on the study of fractional elliptic systems. \\

 \noi Our work is motivated by the work of Tarantello \cite{tarantello} where author used the structure of associated Nehari manifold to obtian the multiplicity of solutions  for the following nonhomogeneous Dirichlet problem on bounded domain $\Om$
 \begin{align*}
 -\De u = |u|^{2^*-2}u+f \;\text{in}\; \Om,\; u=0 \;\text{on}\; \partial \Om
 \end{align*}
Concerning the nonhomogeneous system, Wang et. al  \cite{wang} studied the problem $(P)$ in the local case $s=1$ and obtained a partial multiplicity results. In this paper, we improve their results and show the multiplicity results for $f_1$ and $f_2$ satisfying a weaker assumption \eqref{star0} below.  We also cite \cite{XXW} where multiplicity of positive solutions for nonhomogeneous Choquard equation has been shown using Nehari manifold. We need the following function spaces:

\noi Let us consider the Banach space
\[Y_i:= \left\{u \in W^{s,p}(\mb R^n): \; \int_{\mb R^n}a_i(x)|u|^p~dx < +\infty \right\}\]
equipped with the norm
\[\|u\|_{Y_i}^p = \int_{\mb R^n}\int_{\mb R^n}\frac{|u(x)-u(y)|^p}{|x-y|^{n+sp}}dxdy+ \int_{\mb R^n}a_i(x)|u|^pdx,\; i=1,2.  \]
We define the product space
$Y= Y_1 \times Y_2$ which forms a Banach space with the norm
\[\|(u,v)\|^p := \|u\|_{Y_1}^p+ \|v\|_{Y_2}^p, \]
for all $(u,v)\in Y$.
Throughout this paper, we  assume the following condition on $a_i$, for $i=1,2$
\begin{equation*}\label{cond-on-lambda}
(A)\;\; a_i \in C(\mb R^n,\mb R),\; a_i >0 \; \text{and there exists}\; M_i>0 \; \text{such that}\; \mu(\{x\in \mb R^n: a_i \leq M_i\})< \infty.
\end{equation*}
Then under the condition (A) on $a_i$, for $i=1,2$, we get $Y_i \hookrightarrow L^r(\mb R^n)$ continuously for $r \in [p,p^*_s]$.
To obtain our result, we assume the following condition on perturbation terms:
 \begin{equation}\label{star0}
\int_{\mb R^n} (f_1u+ f_2 v)< C_{p,q}\left(\frac{2q+p-1}{4pq}\right)\|(u,v)\|^{\frac{p(2q-1)}{2q-p}}
\end{equation}
for all $(u,v)\in Y$ such that $$\int_{\mb R^n}\left(\alpha(|x|^{-\mu}*|u|^q)|u|^q +2\beta (|x|^{-\mu}*|u|^q)|v|^q
+\gamma (|x|^{-\mu}*|v|^q)|v|^q \right)dx= 1$$ and
\[C_{p,q}= \left(\frac{p-1}{2q-1}\right)^{\frac{2q-1}{2q-p}}\left(\frac{2q-p}{p-1}\right).\]
It is easy to see that
\[2q > p\left(\frac{2n-\mu}{n}\right) > p-1> \frac{p-1}{2p-1} \]
which implies
\[\frac{2q+p-1}{4pq}<1.\]
So \eqref{star0} implies that
\begin{equation}\label{star00}
\int_{\mb R^n} (f_1u+ f_2 v)< C_{p,q}\|(u,v)\|^{\frac{p(2q-1)}{2q-p}}
\end{equation}
which we will use more frequently rather than our actual assumption \eqref{star0}. The importance of the assumption \eqref{star0} instead of \eqref{star00} can be felt in Lemma \ref{inf-pos-neg}. If $f_1,f_2=0$, then we always have a solution for $(P)$ that is the trivial solution. Now, the main results of this paper goes as follows.
 \begin{Theorem}\label{mainthrm}
Suppose
 $\displaystyle\frac{p}{2}\left(\frac{2n-\mu}{n}\right)< q < \displaystyle\frac{p}{2}\left(\frac{2n-\mu}{n-sp}\right)$,
  $\mu \in (0,n)$ and $(A)$ holds true. Let $0 \not \equiv f_1,f_2 \in L^{\frac{q}{q-1}}(\mb R^n)$ satisfies \eqref{star0}
  then $(P)$ has a weak solution which is a local minimum of $J$ on $Y$.
   Moreover if $f_1,f_2 \geq 0$ then this solution is a nonnegative weak solution.
 \end{Theorem}
 \begin{Theorem}\label{mainthrm2}
Under the hypothesis of Theorem \ref{mainthrm}, $(P)$ has second weak solution $(u_1,v_1)$ in $Y$.
 Also if $f_1,f_2 \geq 0$ then the second solution is non negative.
 \end{Theorem}

\noi This article is organized as follows: In section 2, we set up our function space where our weak solution lies and recall some important results especially the Hardy-Littlewood-Sobolev inequality. In section 3, we analyze fibering maps while defining the  Nehari manifold and show that minimization of energy functional on suitable subsets of Nehari manifold gives us the weak solution to $(P)$. We study the Palais Smale sequences in section 4. Finally, we prove our main theorem in section 5.

\section{Preliminary results}
In this section, we state some important known results which will be used as tools to prove our main results. The key inequality is the following classical Hardy-Littlewood-Sobolev inequality \cite{lieb}.
 \begin{Proposition}\label{HLS}
(\textbf {Hardy-Littlewood-Sobolev inequality}) Let $t,r>1$ and $0<\mu<n $ with $1/t+\mu/n+1/r=2$, $f \in L^t(\mathbb R^n)$ and $h \in L^r(\mathbb R^n)$. There exists a sharp constant $C(t,n,\mu,r)$, independent of $f,h$ such that
 \begin{equation*}
 \int_{\mb R^n}\int_{\mb R^n} \frac{f(x)h(y)}{|x-y|^{\mu}}\mathrm{d}x\mathrm{d}y \leq C(t,n,\mu,r)\|f\|_{L^t(\mb R^n)}\|h\|_{L^r(\mb R^n)}.
 \end{equation*}
 \end{Proposition}
 \begin{Remark}
 In general, let $f = h= |u|^q$ then by Hardy-Littlewood-Sobolev inequality we get,
\[ \int_{\mb R^n}\int_{\mb R^n} \frac{|u(x)|^q|u(y)|^q}{|x-y|^{\mu}}\mathrm{d}x\mathrm{d}y\]
 is well defined if $|u|^q \in L^t(\mb R^n)$ for some $t>1$ satisfying
 \[\frac{2}{t}+ \frac{\mu}{n}=2.\]
 Since we will be working in the space $ W^{s,p}(\mb R^n)$, by fractional Sobolev Embedding theorems(refer \cite{hitch}), we must have $q t \in [p,p^*_s]$, where $p^*_s = \frac{np}{n-sp}$ i.e.
 \[\frac{p}{2}\left(\frac{2n-\mu}{n}\right)\leq q \leq \frac{p}{2}\left(\frac{2n-\mu}{n-sp}\right).\]
 \end{Remark}
 We define
 \[q_l:= \frac{p}{2}\left(\frac{2n-\mu}{n}\right)\; \text{and}\; q_u:=\frac{p}{2}\left(\frac{2n-\mu}{n-sp}\right).  \]
 Here, $q_l$ and $q_u$ are known as lower and upper critical exponents. We constrain our study only when
 \[\frac{p}{2}\left(\frac{2n-\mu}{n}\right)< q < \frac{p}{2}\left(\frac{2n-\mu}{n-sp}\right).\]

\noi Next result is a basic inequality whose proof can be worked out in similar manner as proof of Proposition 3.2(3.3) of \cite{GhSc}.
\begin{Lemma}\label{ineq}
For $u,v \in L^{\frac{2n}{2n-\mu}}(\mb R^n)$, we have
\[\int_{\mb R^n }\int_{\mb R^n}\frac{|u(x)|^q|v(y)|^q}{|x-y|^\mu}~dxdy \leq \left(\int_{\mb R^n}\int_{ \mb R^n}\frac{|u(x)|^q|u(y)|^q}{|x-y|^\mu}~dxdy \right)^{\frac12}  \left(\int_{\mb R^n}\int_{ \mb R^n}\frac{|v(x)|^q|v(y)|^q}{|x-y|^\mu}~dxdy \right)^{\frac12},\]
where $\mu\in (0,n)$ and $q \in [q_l,q_u]$.
\end{Lemma}

We now prove following lemma which is a version of concentration compactness principle proved in Lemma $2.18$ of \cite{ZR}.
\begin{Lemma}\label{conc-comp-lem}
Let $n>sp$ and assume $\{u_k\}$ is bounded in $Y_i$, $i=1,2$ such that it satisfies
\[\lim_{k\to \infty} \sup_{y \in \mb R^n}\int_{B_R(y)}|u_k|^p~dx =0,\]
where $R >0$ and $B_R(y)$ denotes the ball centered at $y$ with radius $R$. Then $u_k \to 0$ strongly in $L^r(\mb R^n)$ for $r \in (p, p^*_s)$, as $k \to \infty$.
\end{Lemma}
\begin{proof}
We prove the result for $i=1$ and for $i=2$, it follows similarly. Let $r \in (p, p^*_s)$, $y \in \mb R^n$ and $R>0$. By using H\"{o}lder inequality, for each $k$ we get
\[\|u_k\|_{L^r(B_R(y))} \leq \|u_k\|^{1-\la}_{L^p(B_R(y))} \|u_k\|^{\la}_{L^{p^*_s}(B_R(y))},\]
where $\frac{1}{r}= \frac{1-\la}{p} +\frac{\la}{p^*_s}$.
Then
\begin{equation}\label{ccl1}
\int_{B_R(y)}|u_k|^r~dx \leq \|u_k\|^{r(1-\la)}_{L^p(B_R(y))} \|u_k\|^{r\la}_{L^{p^*_s}(\mb R^n)}.
\end{equation}
We choose a family of balls $\{B_R(y_i)\}$ where their union covers $\mb R^n$ and are such that each point of $\mb R^n$ is contained in atmost $m$ such balls (where $m$ is a prescribed integer). Now summing \eqref{ccl1} over this family, we obtain
\begin{equation*}
\|u_k\|_{L^r(\mb R^n)}^r \leq m \sup_{y \in \mb R^n}\left( \int_{B_R(y)}|u_k|^p~dx\right)^{\frac{r(1-\la)}{p}}\|u_k\|^{r\la}_{L^{p^*_s}(\mb R^n)}.
\end{equation*}
Using continuous embedding of $Y_1$ in $L^{p^*_s}(\mb R^n)$ and our hypothesis, we get $u_k \to 0$ strongly in $L^r (\mb R^n)$ as $k \to \infty$. \hfill{ \QED}
\end{proof}

 \noi The following is a compactness result for the space $Y_i$, $i=1,2$ which will be used further in our work.
 \begin{Lemma}\label{conc-comp}
 Suppose (A) holds. Then $Y_i$ is compactly embedded in $L^r(\mb R^n)$, $r \in [p,p^*_s)$ and $i=1,2$.
 \end{Lemma}
 \begin{proof}
 We prove it for $Y_1$ and for $Y_2$, it follows analogously. Let $\{u_k\} \subset Y_1$ be a bounded sequence, upto a subsequence, we may assume that $u_k \rightharpoonup u_0$ weakly in $Y_1$ as $k \to \infty$. Then $u_k \to u_0$ in $L^r_{\text{loc}}(\mb R^n)$ as $k \to \infty$, for $r \in [p,p^*_s)$. We first prove that $u_k \to u_0$ strongly in $L^p (\mb R^n)$. Suppose $\xi_k := \|u_k\|_{L^p(\mb R^n)}$ and $\xi_k \to \xi = \|u_0\|_{L^p(\mb R^n)} $, along a subsequence, as $k\to \infty$. So, $\xi \geq \|u_0\|_{L^p(\mb R^n)}$. We claim that for each $\e >0$, there exist $R >0$ such that
 \[\int_{\mb R^n \setminus B_R(0)}|u_k|^p ~dx < \e \; \text{uniformly in } k.\]
  If this holds then $ u_k \to 0$ strongly in $L^p(\mb R^n)$. Because we already have $u_k|_{B_R(0)} \to 0$ strongly in $L^p(B_R(0))$ as $k \to \infty$ which gives
  \begin{equation*}
  \begin{split}
 \xi\geq \|u_0\|_{L^p(\mb R^n)} &= \|u_0\|_{L^p(B_R(0))}+ \|u_0\|_{L^p(\mb R^n \setminus B_R(0))}\\
  & \geq \lim_{k\to \infty} \|u_k\|_{L^p(B_R(0))}\\
  & \geq \lim_{k\to \infty} \|u_k\|_{L^p(\mb R^n)} - \lim_{k\to \infty} \|u_k\|_{L^p(\mb R^n \setminus B_R(0))}\geq \xi -\e .
  \end{split}
  \end{equation*}
 To prove our claim, let us fix $\e >0$ and choose constants $M,C>0$ such that
 \[ M> \frac{2}{\e}\sup_k \|u_k\|^p_{Y_1}\; \text{and}\; C\geq \sup_{ u\in Y_1\setminus \{0\}}\frac{\|u_k\|_{L^{pr(\mb R^n)}}^p}{\|u_k\|_{Y_1}^p}.\]
 Let $r^\prime $ be such that $1/r+1/r^\prime =1$. Now condition (A) implies for $R>0$ large enough,
 \[\mu \left(\left\{x \in \mb R^n \setminus B_R(0): \; a_1(x)< M\right\} \right)\leq \left(\frac{\e}{2C\sup_k\|u_k\|_{Y_1}^p}\right)^{r^\prime}.\]
 We set
 $A= \{x \in \mb R^n \setminus B_R(0): \; a_1(x)\geq M\}$ and
 $B= \{x \in \mb R^n \setminus B_R(0): \; a_1(x)< M\}$.
 Then, we get
 \[\int_A |u_k|^p~dx \leq \int_A \frac{a_1(x)}{M}|u_k|^p~dx \leq \frac{1}{M}\|u_k\|_{Y_1}^p\leq \frac{\e}{2}.\]
 Also using H\"{o}lder's inequality, we get
 \begin{equation*}
 \int_{B}|u_k|^p~dx \leq \left( \int_{B}|u_k|^{pr}~dx\right)^{\frac{1}{r}} (\mu(B))^{\frac{1}{r^\prime}}
 \leq C\|u_k\|_{Y_1}^p (\mu(B))^{\frac{1}{r^\prime}}\leq \frac{\e}{2}.
 \end{equation*}
 Therefore we can write
 \[\int_{\mb R^n \setminus B_R(0)}|u_k|^p~dx = \int_A |u_k|^p~dx  + \int_B |u_k|^p~dx \leq \e\]
  which implies $u_k \to u_0$ strongly in $L^p(\mb R^n)$. Finally, using Lemma \ref{conc-comp-lem}, it follows that $u_k \to u_0$ strongly in $L^r(\mb R^n)$, for $ r \in [p,p^*_s)$. This establishes the proof.\hfill{\QED}
  \end{proof}

 \noi For our convenience, if $u,\phi \in W^{s,p}(\mb R^n)$, we use the notation $\langle u,\phi\rangle $ to denote
\[\langle u,\phi\rangle := \int_{\mb R^n}\int_{\mb R^n} \frac{(u(x)-u(y))|u(x)-u(y)|^{p-2}(\phi(x)-\phi(y))}{|x-y|^{n+sp}}dxdy. \]
\begin{Definition}
A pair of functions $(u,v)\in Y$ is said to be a weak solution to $(P)$ if
\begin{equation*}\label{def-weak-sol}
\begin{split}
&\langle u,\phi_1\rangle + \int_{\mb R^n}a_1(x)u|u|^{q-2}\phi_1~dx+ \langle v,\phi_2\rangle + \int_{\mb R^n}a_2(x)v|v|^{q-2}\phi_2~dx\\
 & \quad  -\alpha \int_{\mb R^n}(|x|^{-\mu}*|u|^q)u|u|^{q-2}\phi_1 ~dx-\gamma \int_{\mb R^n}(|x|^{-\mu}*|v|^q)v|v|^{q-2}\phi_2 ~dx\\
 & \quad \quad-\beta \int_{\mb R^n}(|x|^{-\mu}*|v|^q)u|u|^{q-2}\phi_1~ dx-\beta \int_{\mb R^n}(|x|^{-\mu}*|u|^q)v|v|^{q-2}\phi_2 ~dx\\
 &\quad \quad \quad- \int_{\mb R^n}(f_1\phi_1 +f_2\phi_2)~dx=0,
\end{split}
\end{equation*}
for all $(\phi_1,\phi_2) \in Y$.
\end{Definition}
Thus we define the energy functional corresponding to $(P)$ as
\begin{equation*}
\begin{split}
J(u,v) &= \frac{1}{p}\|(u,v)\|^p - \frac{1}{2q}\int_{\mb R^n}\left(\alpha(|x|^{-\mu}*|u|^q)|u|^q +\beta (|x|^{-\mu}*|u|^q)|v|^q \right)dx\\
&\quad - \frac{1}{2q}\int_{\mb R^n}\left( \beta (|x|^{-\mu}*|v|^q)|u|^q+ \gamma (|x|^{-\mu}*|v|^q)|v|^q \right)dx
-\int_{\mb R^n } (f_1u+f_2v)dx.
\end{split}
\end{equation*}
It is clear that weak solutions to $(P)$ are critical points of $J$. We have the following symmetric property
\[\int_{\mb R^n}(|x|^{\mu}*|u|^q)|v|^q~dx =\int_{\mb R^n}\int_{\mb R^n}\frac{|u(x)|^q|v(y)|^q}{|x-y|^\mu}~dxdy= \int_{\mb R^n}(|x|^{-\mu}*|v|^q)|u|^q~dx.\]
Therefore $J$ can be written as
\begin{equation*}
\begin{split}
J(u,v) &= \frac{1}{p}\|(u,v)\|^p - \frac{1}{2q}\int_{\mb R^n}\left(\alpha(|x|^{-\mu}*|u|^q)|u|^q +2\beta (|x|^{-\mu}*|u|^q)|v|^q + \gamma (|x|^{-\mu}*|v|^q)|v|^q \right)dx\\
&\quad -\int_{\mb R^n } (f_1u+f_2v)dx.
\end{split}
\end{equation*}
In the context of Hardy- Littlewood-Sobolev inequality i.e. Proposition \ref{HLS}, we get
\begin{equation}\label{ineq1}
\int_{\mb R^n}\int_{\mb R^n}\frac{|u(x)|^q|u(y)|^q}{|x-y|^\mu}dxdy \leq  C\|u\|_{L^{\frac{2nq}{2n-\mu}}(\mb R^n)}^{2q},
\end{equation}
for any $u^q \in L^r(\mb R^n)$, $r >1$, $\mu \in (0,n)$ and $q \in [q_l,q_u]$. Using \eqref{ineq1}, Lemma \ref{ineq} and $f_1, f_2 \in L^{\frac{q}{q-1}}(\mb R^n)$, we conclude that $J$ is well defined. Moreover, it can be shown that $J \in C^2(Y,\mb R)$.\\
\textbf{Notations:}\\
1. For any two elements $(z_1,z_2), (w_1,w_2)\in Y$, we define $(z_1,z_2)-(w_1,w_2) := (z_1-w_1,z_2-w_2)$.\\
2. For $t \in \mb R$ and $(u,v)\in Y$, $t(u,v)= (tu,tv)$.

\section{Nehari manifold and Fibering map analysis}
To find the critical points of $J$, we constraint our functional $J$ on the Nehari manifold
\[\mc N = \{(u,v)\in Y: \; (J^\prime(u,v),(u,v)) =0 \},\]
where
\begin{equation*}
\begin{split}
(J^\prime(u,v),(u,v)) &= \|(u,v)\|^p - \int_{\mb R^n}\left(\alpha(|x|^{-\mu}*|u|^q)|u|^q +2\beta (|x|^{-\mu}*|u|^q)|v|^q\right.\\
&\quad \quad \left. + \gamma (|x|^{-\mu}*|v|^q)|v|^q \right)dx
-\int_{\mb R^n } (f_1u+f_2v)dx.
\end{split}
\end{equation*}
Clearly, every nontrivial weak solution to $(P)$ belongs to $\mc N$. Denote $I(u,v)= (J^\prime (u,v),(u,v))$ and subdivide the set $\mc N$ into three sets as follows:
\begin{equation*}
\mc N^\pm = \{(u,v)\in \mc N: \; (I^\prime (u,v),(u,v))\gtrless 0\}\; \text{and}\;
\mc N^0 = \{(u,v)\in \mc N: \; (I^\prime (u,v),(u,v))=0\}.
\end{equation*}
Here
\begin{equation*}
\begin{split}
(I^\prime(u,v),(u,v)) = p\|(u,v)\|^p - 2q\int_{\mb R^n} &\left(\alpha(|x|^{-\mu}*|u|^q)|u|^q +2\beta (|x|^{-\mu}*|u|^q)|v|^q\right.\\
&\quad \quad \left. + \gamma (|x|^{-\mu}*|v|^q)|v|^q \right)dx
-\int_{\mb R^n } (f_1u+f_2v)dx.
\end{split}
\end{equation*}
Then $\mc N^0$ contains the element $(0,0)$ and $\mc N^+ \cup \mc N^0$ and $\mc N^- \cup \mc N^0$ are closed subsets of $Y$. In the due course of this paper, we will subsequently give reason to divide $\mc N$ into above subsets. For $(u,v)\in Y$, we define the fibering map $\varphi: (0,\infty) \to \mb R $ as
\begin{equation*}
\begin{split}
\varphi(t) = J(tu,tv) = \frac{t^p}{p}\|(u,v)\|^p - \frac{t^{2q}}{2q}\int_{\mb R^n} &\left(\alpha(|x|^{-\mu}*|u|^q)|u|^q +2\beta (|x|^{-\mu}*|u|^q)|v|^q\right.\\
 &\quad\quad  \left.+ \gamma (|x|^{-\mu}*|v|^q)|v|^q \right)dx
 -t\int_{\mb R^n } (f_1u+f_2v)dx.
\end{split}
\end{equation*}
This gives
\begin{equation*}
\begin{split}
\varphi^\prime(t) = t^{p-1}\|(u,v)\|^p - t^{2q-1}\int_{\mb R^n}&\left(\alpha(|x|^{-\mu}*|u|^q)|u|^q +2\beta (|x|^{-\mu}*|u|^q)|v|^q \right.\\
&+ \left.\gamma (|x|^{-\mu}*|v|^q)|v|^q \right)dx -\int_{\mb R^n } (f_1u+f_2v)dx
\end{split}
\end{equation*}
\begin{equation*}
\begin{split}
\text{and}\; \varphi^{\prime\prime}(t)= (p-1)t^{p-2}\|(u,v)\|^p - (2q-1)t^{2q-2}\int_{\mb R^n}&\left(\alpha(|x|^{-\mu}*|u|^q)|u|^q +2\beta (|x|^{-\mu}*|u|^q)|v|^q\right.\\
&\quad \left. + \gamma (|x|^{-\mu}*|v|^q)|v|^q \right)dx.
\end{split}
\end{equation*}
It is easy to see that $(tu,tv)\in \mc N$ if and only if $\varphi^\prime(t)=0$, for $t>0$
i.e
\[\mc N=\{(tu,tu)\in Y: \; \varphi^\prime(t)=0\}.\]
Also, we can check that for $(tu,tv)\in \mc N$, $(I^\prime(tu,tv),(tu,tv))>$ or $<0$ if and only if $\varphi^{\prime\prime}(t)> $ or $<0$ respectively. Therefore, $\mc N^+$, $\mc N^-$ and $\mc N^0$ can also be written as
\begin{equation*}
\mc N^\pm = \{ (tu,tv)\in \mc N:\; \varphi^{\prime\prime}(t)\gtrless 0 \},\\
\text{ and }\mc N^0 = \{ (tu,tv)\in \mc N:\; \varphi^{\prime\prime}(t)=0 \}.
\end{equation*}
\begin{Lemma}\label{J-bdd-below}
If $f_1,f_2 \in L^{\frac{q}{q-1}}(\mb R^n)$, then $J$ is coercive and bounded below on $\mc N$. Hence $J$ is bounded below on $\mc N^+$ and $\mc N^-$.
\end{Lemma}
\begin{proof}
Let $(u,v) \in \mc N$, then $(J^\prime(u,v),(u,v))=0$ i.e.
\[\|(u,v)\|^p - \int_{\mb R^n}(f_1u+f_2v)~dx = L(u,v).\]
Using this we obtain
\begin{equation*}
\begin{split}
J(u,v) &= \left( \frac{2q-p}{2qp}\right)\|(u,v)\|^p - \left( \frac{2q-1}{2q}\right)\int_{\mb R^n}(f_1u+f_2v)~dx\\
&\geq \left( \frac{2q-p}{2qp}\right)\|(u,v)\|^p - \left(\frac{2q-1}{2q}\right) \|f_1\|_{L^{\frac{q}{q-1}}(\mb R^n)}\|u\|_{L^q(\mb R^n)}+ \|f_2\|_{L^{\frac{q}{q-1}}(\mb R^n)}\|v\|_{L^q(\mb R^n)}\\
& \geq \|(u,v)\|\left(\left( \frac{2q-p}{2qp} \right)\|(u,v)\|^{p-1} - \left( \frac{2q-1}{2q}\right) (S_{q,1}+S_{q,2})\right.\\
& \quad \quad \quad \quad \quad  \left.\max\left\{\|f_1\|_{L^{\frac{q}{q-1}}(\mb R^n)},\|f_2\|_{L^{\frac{q}{q-1}}(\mb R^n)}\right\}\right),
\end{split}
\end{equation*}
where $S_{q,i}$ denotes the best constant for the embedding $Y \hookrightarrow L^q(\mb R^n)$, $i=1,2$. This implies that $J$ is coercive and bounded below on $\mc N$. \hfill{\QED}
\end{proof}

\noi Thus it is natural to consider a minimization problem on $\mc N$ or its subsets. We fix $(u,v)\in Y$ and define
\[m(t):= t^{p-1}\|(u,v)\|^p - t^{2q-1}\int_{\mb R^n}\left(\alpha(|x|^{-\mu}*|u|^q)|u|^q +2\beta (|x|^{-\mu}*|u|^q)|v|^q
+\gamma (|x|^{-\mu}*|v|^q)|v|^q \right)dx,\]
\[ K:= K(u,v)= \int_{\mb R^n } (f_1u+f_2v)dx \]
\[\text{and}\; L = L(u,v)= \int_{\mb R^n}\left(\alpha(|x|^{-\mu}*|u|^q)|u|^q +2\beta (|x|^{-\mu}*|u|^q)|v|^q
+\gamma (|x|^{-\mu}*|v|^q)|v|^q \right)dx.  \]
Then $\varphi^\prime(t)=0$ iff $m(t)= K$. Since $p\left(\frac{2n-\mu}{n}\right) < 2q$ and $\left(\frac{2n-\mu}{n}\right)>1$, we get $p<2q$ which implies $\lim\limits_{t\to +\infty}m(t)=-\infty$. Also $\lim\limits_{t\to0}m(t)=0$ and it is easy to check that
\[t_0= \left(\frac{(p-1)\|(u,v)\|^p }{(2q-1)L} \right)^{\frac{1}{2q-p}} \]
 is a point of global maximum for $m(t)$. For $t>0$ small enough, $m(t)>0$. Altogether, this implies that if we choose $K>0$ sufficiently small then $m(t)=K$ is satisfied in such a way that $\varphi^\prime(t)=0$ has two positive solutions $t_1,t_2$ such that $0<t_1<t_0<t_2$. Then according to the sign of $\varphi^{\prime\prime}(t_1)$ and $\varphi^{\prime \prime}(t_2)$, we decide in which subset(i.e  $\mc N^+, \mc N^-, \mc N^0$) they lie. Hence the sets $\mc N^+, \mc N^-$ and $ \mc N^0$  contains the point of local maximum, local minimum and point of inflexion of the fibering maps.

\noi We end this section with the following two results.
\begin{Lemma}\label{N0-empty}
If $f_1,f_2 \in L^{\frac{q}{q-1}}(\mb R^n)$ are non zero and satisfies \eqref{star0}, then $\mc N^0 =\{(0,0)\}$.
\end{Lemma}
\begin{proof}
To prove that $\mc N^0 = \{(0,0)\}$, we need to show that for $(u,v)\in Y\setminus \{(0,0)\}$, $\varphi(t)$ has no critical point which is a saddle point. Let $(u,v) \in Y\setminus \{(0,0)\}$. From above analysis, we know that $m(t)$ has unique point of global maximum at $t_{0}$ and
\[m(t_0)= \left( \frac{2q-p}{2q-1}\right)\left(\frac{p-1}{L(2q-1)}\right)^{\frac{p-1}{2q-p}}\|(u,v)\|^{\frac{2q-1}{2q-p}}.\]
From the analysis of the map $m(t)$ done above, we get that if $0<K< m(t_0)$, then $\varphi^\prime(t)=0$ has exactly two roots $t_1,t_2$ such that $0<t_1<t_0<t_2$ and if $K\leq 0$ then $\varphi^\prime(t)=0$ has only one root $t_3$ such that $t_3> t_0$. Since $\varphi^{\prime \prime}(t)= m^\prime(t)$, we get $\varphi^{\prime \prime}(t_1)>0$, $\varphi^{\prime \prime}(t_2)<0$ and $\varphi^{\prime \prime}(t_3)<0$. Hence if $0<K<m(t_0)$ then $(t_1u, t_1v)\in \mc N^+$, $(t_2u,t_2v)\in \mc N^-$ and if $K\leq 0$ then $(t_3u,t_3v)\in \mc N^-$. This implies $\{ (u,v)\in Y:\;  0<K<m(t_0)\}\cap \mc N^\pm \neq \emptyset$ and  $\{ (u,v)\in Y:\; K\leq 0\}\cap \mc N^- \neq \emptyset$. As a consequence, $\mc N^\pm \neq \emptyset$. We saw that for any sign of $K$, critical point of $\varphi(t)$ is either a point of local maximum or local minimum which implies $\mc N^0 = \{(0,0)\}$. It remains to show that $0< K< m(t_0)$ holds. But that is clearly implied by the condition \eqref{star0} which we have already assumed. This completes the proof.
 \hfill{\QED}
\end{proof}

\begin{Lemma}\label{N-closed}
 Assume $f_1,f_2 \in L^{\frac{q}{q-1}}(\mb R^n)$ satisfies \eqref{star0}, then $\mc N^-$ is closed.
\end{Lemma}
\begin{proof}
Let $cl(\mc N^-)$ denotes the closure of $\mc N^-$. Since $cl(\mc N^-) \subset \mc N^- \cup \{(0,0)\}$, it is enough to show that $(0,0) \not\in \mc N^-$ or equivalently $dist((0,0),\mc N^-)>0$. We denote $(\hat u,\hat v) = \frac{(u,v)}{\|(u,v)\|}$ for $(u,v)\in \mc N^-$, then $\|(\hat u, \hat v)\|=1$. Let us consider the fibering map $\varphi(t)$ corresponding to $(\hat u, \hat v)$. From proof of Lemma \ref{N0-empty}, we get that if $K\leq 0$ then $\varphi^\prime (t)=0$ has exactly one root $t_3>t_0$ such that $(t_3\hat u, t_3 \hat v) \in \mc N^-$. If $(t_3\hat u,t_3\hat v)= (u,v) \in \mc N^-$, then $t_3=\|(u,v)\|$. Also, if $0<K<m(t_0)$ then $\varphi^\prime (t)=0$ has exactly two roots $t_1,t_2$ satisfying $t_1<t_0<t_2$ such that $(t_1\hat u,t_1\hat v)\in \mc N^+$ and $(t_2\hat u,t_2\hat v)\in \mc N^-$. Hence if $(t_2\hat u,t_2\hat v) = (u,v)\in \mc N^-$ then $t_2 = \|(u,v)\|$. Since $t_2,t_3>t_0$, we get
$\|(u,v)\| > t_0$. Using Lemma \ref{ineq}, inequality \ref{ineq1}, continuous embedding of $Y$ in $L^r(\mb R^n)$ for $r \in [p,p^*_s]$, $\frac{2nq}{2n-\mu} \in (p,p^*_s)$ and $\|(\hat u,\hat v)\|=1$, we get that
\begin{equation}\label{N-closed1}
\begin{split}
L &\leq C\left(\alpha \|\hat u\|_{L^{\frac{2nq}{2n-\mu}}(\mb R^n)}^{2q} + \gamma \|\hat v\|_{L^{\frac{2nq}{2n-\mu}}(\mb R^n)}^{2q}+ 2\beta\left( \|\hat u\|_{L^{\frac{2nq}{2n-\mu}}(\mb R^n)}^{2q} \right)^{\frac12} \left( \|\hat v\|_{L^{\frac{2nq}{2n-\mu}}(\mb R^n)}^{2q} \right)^{\frac12}\right)\\
&\leq C_1 \left( \alpha\|\hat u\|_{Y_1}^{2q}+ \gamma \|\hat v\|_{Y_2}^{2q}+ 2\beta \|\hat u\|_{Y_1}^{q}\|\hat v\|_{Y_2}^{q} \right)\leq C_2 \|(\hat u,\hat v)\|^{2q},
\end{split}
\end{equation}
where $C_1,C_2>0$ are constants independent of $\hat u$ and $\hat v$. This implies $L $ is bounded above on the unit sphere of $Y$. Since $\|\hat u, \hat v\|=1$, from definition of $t_0$ it follows that
\[t_0 \geq \left(\frac{p-1}{(2q-1)\sup_{\|(u,v)\|=1}L(u,v)}\right)^{\frac{1}{2q-p}}:=\theta.\]
 Therefore, $dist((0,0),\mc N^-) = \inf\limits_{(u,v)\in \mc N^-} \{\|(u,v)\|\} \geq \theta >0$ and this proves the Lemma. \hfill{\QED}
\end{proof}
\noi Using Lemma \ref{J-bdd-below}, we can define the following
\begin{equation}\label{min-N+}
\Upsilon^+ := \inf\limits_{(u,v) \in \mc N^+}J(u,v),\; \text{and}\;
\end{equation}
\begin{equation}\label{min-N-}
\Upsilon^- := \inf\limits_{(u,v) \in \mc N^-}J(u,v).
\end{equation}
If the infimum in the above two equations are achieved, then we can show that they form a weak solution of our problem $(P)$.
\begin{Lemma}\label{exist-sol}
Let $(u_1,v_1)$ and $(u_2,v_2)$ are minimizers of $J$ on $\mc N^+$ and $\mc N^-$ respectively. Then $(u_1,v_1)$ and $(u_2,v_2)$ are nontrivial weak solutions of $(P)$.
\end{Lemma}
\begin{proof}
Let $(u_1,v_1) \in \mc N^+$ such that $J(u_1,v_1)= \Upsilon^+$ and define $V:= \{(u,v)\in Y: \; (I^\prime(u,v),(u,v))>0\}$. So, $\mc N^+ = \{(u,v)\in V:\; J(u,v)=0\}$. Using Theorem $4.1.1$ of \cite{chang} we deduce that there exists Lagrangian multiplier $\la \in \mb R$ such that
\[J^\prime(u_1,v_1)= \la I^\prime(u_1,v_1).\]
Since $(u_1,v_1) \in \mc N^+$, $(J^\prime(u_1,v_1),(u_1,v_1))=0$ and $(I^\prime(u_1,v_1),(u_1,v_1))>0$. This implies $\la=0$. Therefore, $(u_1,v_1)$ is a nontrivial weak solution of $(P)$. Similarly, we can prove that if $(u_2,v_2) \in \mc N^-$ is such that $J(u_2,v_2)= \Upsilon^-$ then $(u_2,v_2)$ is also a nontrivial weak solution of $(P)$.\hfill{\QED}
\end{proof}

\noi Our next result is an observation regarding the minimizers $\Upsilon^+$ and $\Upsilon^-$.
\begin{Lemma}\label{inf-pos-neg}
If $0\neq f_1,f_2 \in L^{\frac{q}{q-1}}(\mb R^n)$ satisfies \eqref{star0} then $ \Upsilon^->0$ and  $\Upsilon^+ < 0 $.
\end{Lemma}
\begin{proof}
Let $(u,v)\in Y$ then from the proof of Lemma \ref{N0-empty}, we know that if $f_1,f_2$ satisfies \eqref{star0} then $K< m(t_0)$. In that case if $0<K<m(t_0)$
 then corresponding to $(u,v)$, $\varphi^\prime(t)=0$ has exactly two roots $t_1$ and $t_2$ such that $t_1<t_0<t_2$, $t_1(u,v)\in \mc N^+$ and $t_2(u,v)\in \mc N^-$. Since $\varphi^\prime(t) = \|(u,v)\|t^{p-1} - Lt^{2q-1}-K$, $\lim\limits_{t\to 0^+}\varphi^\prime(t)= -K <0$. Also $\varphi^{\prime \prime}(t)>0$ for all $t \in (0,t_0)$. Since $t_1$ is point of local minimum of $\varphi(t)$, $t_1>0$ and $\lim\limits_{t\to 0^+}\varphi(t)=0$, we get $\varphi(t_1)<0$. Therefore,
\[0> \varphi(t_1)= J(t_1u,t_1v)\geq \Upsilon^+.\]
Now we prove that $\Upsilon^->0$. From \eqref{N-closed1}, we know that $L\leq C_2\|(u,v)\|^{2q}$. This implies that there exists a constant $C_3>0$ which is independent of $(u,v)$ such that
\[\frac{(\|(u,v)\|^p)^{\frac{2q}{2q-p}}}{L^{\frac{p}{2q-p}}}\geq C_3.\]
Now using this and the given hypothesis we consider $\varphi(t_0)$ corresponding to $(u,v)$ as
\begin{equation*}\label{inf-pos-neg1}
\begin{split}
\varphi(t_0) &= \frac{t_0^p}{p}- L\frac{t_0^{2q}}{2q}-Kt_0\\
& =\frac{1}{p}\left(\frac{(p-1)\|(u,v)\|^p }{(2q-1)L} \right)^{\frac{p}{2q-p}}- \frac{L
}{2q}{\left(\frac{(p-1)\|(u,v)\|^p }{(2q-1)L} \right)^{\frac{2q}{2q-p}}}-K \left(\frac{(p-1)\|(u,v)\|^p }{(2q-1)L} \right)^{\frac{1}{2q-p}}
\end{split}
\end{equation*}
\begin{equation*}
\begin{split}
& = \frac{(2q-p)(2q+p-1)}{2qp(2q-1)}\left( \frac{p-1}{2q-1}\right)^{\frac{p}{2q-p}}\frac{(\|(u,v)\|^p)^{\frac{2q}{2q-p}}}{L^{\frac{p}{2q-p}}}- K \left( \frac{p-1}{2q-1}\right)^{\frac{1}{2q-p}} \frac{(\|(u,v)\|^p)^{\frac{1}{2q-p}}}{L^{\frac{1}{2q-p}}}\\
& \geq \left(\frac{(2q-p)(2q+p-1)(p-1)^{\frac{p}{2q-p}}}{2qp(2q-1)^{\frac{2q}{2q-p}}} \right)\frac{(\|(u,v)\|^p)^{\frac{2q}{2q-p}}}{L^{\frac{p}{2q-p}}}\\
& \geq C_3  \left(\frac{(2q-p)(2q+p-1)(p-1)^{\frac{p}{2q-p}}}{2qp(2q-1)^{\frac{2q}{2q-p}}} \right):= M(\text{say}).
\end{split}
\end{equation*}
Hence
\begin{equation*}
\begin{split}
\Upsilon^-= \inf_{(u,v)\in Y\setminus \{(0,0)\}}\max_{t}J(tu,tv) \geq \inf_{(u,v)\in Y\setminus \{(0,0)\}}\varphi(t_0)\geq M >0.
\end{split}
\end{equation*}
which completes the proof.\hfill{\QED}
\end{proof}

\section{Palais-Smale Analysis}
In this section, we study the nature of minimizing sequences for the functional $J$. First we prove some lemmas which will assert the existence of Palais Smale sequence for the minimizer of $J$ on $\mc N$. The following lemma is a consequence of Lemma \ref{N0-empty}.

\begin{Lemma}\label{IFT}
Let $0\neq f_1,f_2 \in L^{\frac{q}{q-1}}(\mb R^n)$ satisfies \eqref{star0}. Given $(u,v)\in \mc N\setminus \{(0,0)\}$ ,there exist $\e>0$ and a differentiable function $\Im : B((0,0),\e) \subset Y \to \mb R^+:=(0,+\infty)$ such that $\Im(0,0)=1$, $\Im(w_1,w_2)((u,v)-(w_1,w_2))\in \mc N$ and
\begin{equation}\label{IFT1}
(\Im^\prime(0,0),(w_1,w_2))= \frac{p(A(w_1,w_2)+ A_2(w_1,w_2))- q R(w_1,w_2)-\int_{\mb R^n} (f_1w_1+f_2w_2)}{(p-1)\|(u,v)\|^p - (2q-1)L(u,v)},
\end{equation}
for all $(w_1,w_2)\in B((0,0),\e)$ where $A_1(w_1,w_2):=\langle u,w_1\rangle+ \int_{\mb R^n}a_1(x)|u|^{p-2}uw_1$, $A_2(w_1,w_2):= \langle v,w_2\rangle+ \int_{\mb R^n}a_2(x)|v|^{p-2}vw_2 $ and
\begin{align*}
R(w_1,w_2) := &2\alpha \int_{\mb R^n} (|x|^{-\mu}* |u|^q)|u|^{q-2}uw_1 + 2\gamma \int_{\mb R^n} (|x|^{-\mu}* |v|^q)|v|^{q-2}vw_2 \\
& \quad + \beta \int_{\mb R^n} (|x|^{-\mu}* |u|^q)|v|^{q-2}vw_2+ \beta \int_{\mb R^n} (|x|^{-\mu}* |v|^q)|u|^{q-2}uw_1.
\end{align*}
\end{Lemma}
\begin{proof}
Fixing a function $(u,v)\in \mc N$, we define the map $F: \mb R \times Y \to \mb R$ as follows
\[F(t, (w_1,w_2)):= t^{p-1}\|(u,v)-(w_1,w_2)\|-t^{2q-1}L((u,v)-(w_1,w_2))-\int_{\mb R^n} (f_1(u-w_1)+f_2(v-w_2)).\]
It is easy to see that $F$ is differentiable. Since $F(1,(0,0))=(J^\prime (u,v),(u,v))=0$ and $F_t(1,(0,0))= (p-1)t^{p-2}\|(u,v)-(w_1,w_2)\|^p - (2q-1)t^{2q-2}L((u,v)-(w_1,w_2)) \neq 0$ by Lemma \ref{N0-empty}, we apply the Implicit Function Theorem at the point $(1,(0,0))$ to get the existence of $\e>0$ and a differentiable function $\Im : B((0,0),\e) \subset Y \to \mb R^+$ such that
\[\Im(0,0)=1 \; \text{and}\; F((w_1,w_2), \Im(w_1,w_2))=0, \; \text{for all}\; (w_1,w_2)\in B((0,0),\e).\]
This implies
\begin{equation*}
\begin{split}
0 &= \Im^{p-1}(w_1,w_2)\|(u,v)-(w_1,w_2)\|^p- \Im^{2q-1}(w_1,w_2)L((u,v)-(w_1,w_2))- K((u,v)-(w_1,w_2))\\
&= \frac{1}{\Im(w_1,w_2)}\left[ \|\Im(w_1,w_2)(u,v)-(w_1,w_2)\|^p - L(\Im(w_1,w_2)((u,v)-(w_1,w_2)))\right.\\
 & \quad \quad \left.- K(\Im(w_1,w_2)((u,v)-(w_1,w_2))) \right].
\end{split}
\end{equation*}
Since $\Im(w_1,w_2) >0$ we get $\Im(w_1,w_2)((u,v)-(w_1,w_2)) \in \mc N$ whenever $(w_1,w_2)\in B((0,0),\e)$. Finally \eqref{IFT1} can be obtained by differentiating $F((w_1,w_2), \Im(w_1,w_2))=0$ with respect to $(w_1,w_2)$. \hfill{\QED}
\end{proof}

Let us define
\[\Upsilon := \inf_{(u,v)\in \mc N}J(u,v).\]

\begin{Lemma}\label{uprbd}
There exist a constant $C_1>0$ such that $\Upsilon \leq - \frac{(2q-p)(2qp-2q-p)}{4pq^2}C_1$.
\end{Lemma}
\begin{proof}
Let $(\hat u,\hat v)\in Y$ be the unique solution to the equations given below
\begin{align*}
(-\De)^s_p\hat u +a_1(x)|\hat u|^{p-1}\hat u & = f_1 \; \text{in}\; \mb R^n\\
(-\De)^s_p\hat v +a_2(x)|\hat v|^{p-1}\hat v & = f_2 \; \text{in}\; \mb R^n.
\end{align*}
So, since $f_1,f_2 \neq 0$
\[\int_{\mb R^n}(f_1 \hat u+f_2 \hat v)=\|(\hat u,\hat v)\|^p>0.\]
Then by Lemma \ref{N0-empty}, we know that there exist $t_1>0$ such that $t_1(\hat u,\hat v)\in \mc N^+$. Consequently
\begin{align*}
J(t_1\hat u, t_1 \hat v) &= -\left(\frac{p-1}{p}\right)t_1^p \|(\hat u,\hat v)\|^p+ \left( \frac{2q-1}{2q}\right)t_1^{2q}L(\hat u,\hat v)\\
& < -\left(\frac{p-1}{p}\right)t_1^p \|(\hat u,\hat v)\|^p+ \frac{p(2q-1)}{4q^2}t_1^p \|(\hat u,\hat v)\|^p\\
&= - \frac{(2q-p)(2qp-2q-p)}{4pq^2}t_1^p \|(\hat u,\hat v)\|^p <0.
\end{align*}
Taking $C_1 =t_1^p \|(\hat u,\hat v)\|^p$ we get the result.\hfill{\QED}
\end{proof}\\
We recall the following Lemma.

\begin{Theorem}\label{Sobolemma}\cite{stein}
Let $0<\theta<n$, $1<r<m<\infty$ and $\frac{1}{m}= \frac{1}{r}-\frac{\theta}{n}$, then
\[\left|\int_{\mb R^n}\frac{f(y)}{|x-y|^{n-\theta}}~dy\right|_{L^{m}(\mb R^n)} \leq C\|f\|_{L^r(\mb R^n)},\]
where $C>0$ is a constant.
\end{Theorem}
 This implies that the Reisz potential defines a linear and continuous map from $L^r(\mb R^n)$ to $L^m(\mb R^n)$, where $r,m$ are defined in above theorem.
\begin{Lemma}\label{inf-achvd}
For $0\neq f_1,f_2 \in L^{\frac{q}{q-1}}(\mb R^n)$,
\[\inf_Q \left( C_{p,q}\|(u,v)\|^{\frac{p(2q-1)}{2q-p}}- \int_{\mb R^n}(f_1u+f_2v)~dx\right):= \delta\]
is achieved, where $ Q = \{(u,v)\in Y : L(u,v)=1\}$. Also if $f_1,f_2$ satisfies \eqref{star0}, then $\delta >0$.
\end{Lemma}
\begin{proof}
Let us define the functional $T: Y \mapsto \mb R$ as
\[T(u,v)=  C_{p,q}\|(u,v)\|^{\frac{p(2q-1)}{2q-p}}- \int_{\mb R^n}(f_1u+f_2v)~dx. \]
This implies
\[T(u,v) \geq C_{p,q}\|(u,v)\|^{\frac{p(2q-1)}{2q-p}} -   (S_{q,1}+S_{q,2})\max \left\{\|f_1\|_{L^{\frac{q}{q-1}}(\mb R^n)},\|f_2\|_{L^{\frac{q}{q-1}}(\mb R^n)} \right\}\|(u,v)\|,\]
where $S_{q,i}$ denotes the best constant for the embedding $Y \hookrightarrow L^q(\mb R^n)$, $i=1,2$. Since $\frac{p(2q-1)}{2q-p} >1$, $T$ is coercive. Let $\{(u_k,v_k)\}\subset Q$ be such that $(u_k,v_k)\rightharpoonup (u,v)$ weakly in $Y$. Then
\begin{equation*}
\begin{split}
\lim_{k\to \infty} \int_{\mb R^n}(f_1u_k+f_2v_k)~dx &= \int_{\mb R^n}(f_1u+f_2v)~dx, \; \text{and}\\
\|(u,v)\|^{\frac{p(2q-1)}{2q-p}} & \leq \liminf_{k\to \infty}\|(u_k,v_k)\|^{\frac{p(2q-1)}{2q-p}}.
\end{split}
\end{equation*}
which implies $T(u,v) \leq \liminf\limits_{k\to \infty} T(u_k,v_k)$ i.e. $T$ is weakly lower semicontinuous. Consider
\begin{equation}\label{conv1}
 \begin{split}
 &\int_{\mb R^n}(|x|^{-\mu}* |u_k|^q)|u_k|^q~dx -  \int_{\mb R^n}(|x|^{-\mu}* |u|^q)|u|^q~dx \\
 &=  \int_{\mb R^n}(|x|^{-\mu}* (|u_k|^q-|u|^q))(|u_k|^q-|u|^k)~dx+ 2\int_{\mb R^n}(|x|^{-\mu}* |u|^q)(|u_k|^q-|u|^k)~dx.
 \end{split}
 \end{equation}
 Since $\frac{2nq}{2n-\mu}< p^*_s$, using Lemma \ref{conc-comp} we have
 \begin{equation}\label{conv2}
 \begin{split}
 |u_k|^q -|u|^q \rightarrow 0
\; \text{in}\; L^{\frac{2n}{2n-\mu}}(\mb R^n) \; \text{as}\; k \rightarrow \infty.
 \end{split}
 \end{equation}
 and thus, using Theorem \ref{Sobolemma} we have
 \begin{equation}\label{conv3}
 |x|^{-\mu}* (|u_k|^q -|u|^q) \rightarrow 0 \; \text{in} \; L^{\frac{2n}{\mu}(\mb R^n)}\; \text{as}\; k \rightarrow \infty.
 \end{equation}
 From \eqref{conv2}, \eqref{conv3} and using H\"{o}lder's inequality in \eqref{conv1}, we get
 \begin{equation}\label{conv4}
 \int_{\mb R^n}(|x|^{-\mu}* |u_k|^q)|u_k|^q~dx \rightarrow \int_{\mb R^n}(|x|^{-\mu}* |u|^q)|u|^q~dx\; \text{as}\; k \rightarrow \infty.
 \end{equation}
 Similarly, we get
  \begin{equation}\label{conv5}
 \int_{\mb R^n}(|x|^{-\mu}* |v_k|^q)|v_k|^q~dx \rightarrow \int_{\mb R^n}(|x|^{-\mu}* |v|^q)|v|^q~dx\; \text{as}\; k \rightarrow \infty.
 \end{equation}
It is easy to see that
 \begin{equation*}
 \begin{split}
  &\int_{\mb R^n}(|x|^{-\mu}* |u_k|^q)|v_k|^q~dx -  \int_{\mb R^n}(|x|^{-\mu}* |u|^q)|v|^q~dx\\
  &=  \int_{\mb R^n}(|x|^{-\mu}* (|u_k|^q-|u|^q))(|v_k|^q-|v|^q)~dx+  \int_{\mb R^n}(|x|^{-\mu}* (|u_k|^q-|u|^q))|v|^q~dx\\
   & \quad+ \int_{\mb R^n}(|x|^{-\mu}* (|v_k|^q-|v|^q))|u|^q~dx.
  \end{split}
 \end{equation*}
  which implies that
  \begin{equation}\label{conv6}
   \int_{\mb R^n}(|x|^{-\mu}* |u_k|^q)|v_k|^q~dx \rightarrow \int_{\mb R^n}(|x|^{-\mu}* |u|^q)|v|^q~dx\; \text{as}\; k \rightarrow \infty.
  \end{equation}
 Thus using \eqref{conv4}, \eqref{conv5} and \eqref{conv6}, we get
$\lim\limits_{k\to \infty} L(u_k,u_k)= L(u,v).$
Since $(u_k,v_k)\in Q$ for each $k$, we get $L(u,v) =1$ which implies $(u,v)\in Q$. Therefore $Q$ is weakly sequentially closed subset of $Y$. Since $Y$ forms a reflexive Banach space, there exists $(u_0,v_0) \in Q$ such that
\[\inf_Q T(u,v) = T(u_0,v_0).\]
Furthermore, it is obvious that if $f_1,f_2$ satisfies \eqref{star0}, then $\delta \geq T(u_0,v_0)>0$.
This establishes the result.\hfill{\QED}
\end{proof}\\
\noi For $(u,v)\in Y\setminus \{(0,0)\}$, we set
\begin{equation*}\label{G(u,v)}
G(u,v):= C_{p,q} \frac{(\|(u,v)\|^p)^{\frac{2q-1}{2q-p}}}{L(u,v)^{\frac{p-1}{2q-p}}}- K(u,v).
\end{equation*}
\begin{Corollary}\label{inf-achvd-cor}
For any $\rho>0$, $\inf\limits_{L(u,v)\geq \rho}G(u,v) \geq \rho\delta$.
\end{Corollary}
\begin{proof}
For $t>0$, if $L(u,v)=1$ for $(u,v)\in Y$ then using Lemma \ref{inf-achvd} we have
\[G(tu,tv)= t \left( C_{p,q}t^{\frac{2q-1}{2q-p}}(\|(u,v)\|^p)^{\frac{2q-1}{2q-p}}- K(u,v)\right) \geq t \delta.\]
This implies for any $\rho >0$, $\inf\limits_{L(u,v)\geq \rho}G(u,v) \geq \rho\delta$ which completes the proof.\hfill{\QED}
\end{proof}

In the next result, we show the existence of P.-S. sequence for $\Upsilon$.

\begin{Proposition}\label{PS-seq}
Let  $0\neq f_1, f_2 \in L^{\frac{q}{q-1}}(\mb R^n)$ such that \eqref{star0}. Then there exists a sequence $(u_k,v_k) \subset \mc N$ such that $J(u_k,v_k) \to \Upsilon $ and $\|J^\prime(u_k,v_k)\|_{Y^*} \to 0$ as $k \to \infty$,  $\|\cdot\|_{Y^*}$ denotes the operator norm on dual of $Y$ i.e. $Y^*$. .
\end{Proposition}
\begin{proof}
From Lemma \ref{J-bdd-below}, we already know that $J$ is bounded from below on $\mc N$. So by Ekeland's Variational principle we get a sequence $\{(u_k,v_k)\} \subset \mc N$ such that
\begin{equation}\label{PS-seq1}
\left\{
\begin{split}
J(u_k,v_k) &\leq \Upsilon +\frac{1}{k},\; \text{and}\;\\
J(u,v) & \geq J(u_k,v_k)- \frac{1}{k}\|(u_k-u,v_k-v)\|, \; \text{for all}\; (u,v)\in \mc N.
\end{split}
\right.
\end{equation}
By taking $k>0$ large enough we have
\[J(u_k,v_k)= \left(\frac{(2q-p)}{2qp}\right)\|(u_k,v_k)\|^p - \left(\frac{(2q-1)}{2qp}\right)\int_{\mb R^n}(f_1u+f_2v)~dx < \Upsilon+\frac{1}{k}. \]
This alongwith  Lemma \ref{uprbd} gives
\begin{equation}
\label{PS-seq2}\int_{\mb R^n}(f_1u+f_2v)~dx \geq \frac{(2q-p)(2qp-2q-p)}{2pq(2q-1)}C_1>0.
\end{equation}
Therefore $u_k, v_k \neq 0$ for all $k$. From \eqref{PS-seq1} and definition of $\Upsilon$, it is clear that $J(u_k,v_k)\to \Upsilon<0$ as $k \to \infty$. Since $\{(u_k,v_k)\} \subset \mc N^+$, we get
\begin{equation}\label{PS-seq3}
\|(u_k,v_k)\|^p - \int_{\mb R^n}(f_1u_k+f_2v_k)~dx = L.
\end{equation}
Using definition of $J$, \eqref{PS-seq1}, \eqref{PS-seq2} and \eqref{PS-seq3}, we get
\begin{equation}\label{PS-seq4}
\begin{split}
\Upsilon^+ + \frac{1}{k}& \geq \left( \frac{1}{p}- \frac{1}{2q}\right)\|(u_k,v_k)\|^p - \left( 1-\frac{1}{2q}\right)\int_{\mb R^n}(f_1u_k+f_2v_k)~dx\\
& \geq \left( \frac{1}{p}- \frac{1}{2q}\right)\|(u_k,v_k)\|^p \\
& \quad \quad - \left( 1-\frac{1}{2q}\right) (S_{q,1}+S_{q,2})\max\left\{\|f_1\|_{L^{\frac{q}{q-1}}(\mb R^n)},\|f_2\|_{L^{\frac{q}{q-1}}(\mb R^n)}\right\}\|(u_k,v_k)\|.
\end{split}
\end{equation}
This implies $\{(u_k,v_k)\}$ is bounded. Now we claim that $\inf_k\|(u_k,v_k)\|\geq \eta>0$, for some constant $\eta$. Suppose not, then, upto a subsequence, $\|(u_k,v_k)\| \to 0$ as $k\to \infty$. This implies $J(u_k,v_k) \to 0$ as $k \to \infty$, using \eqref{PS-seq4} which is a contradiction to  first assertion. So there exist constants $d_1,d_2>0$ such that
\begin{equation}\label{PS-seq5}
d_1 \leq \|(u_k,v_k)\|\leq d_2.
\end{equation}
Now we aim to show that $\|J^\prime(u_k,v_k)\|_{Y^*}\to 0$ as $k \to \infty$. By Lemma \ref{IFT}, for each $k$ we obtain a differentiable function $\Im_k : B((0,0),\e_k) \subset Y \to \mb R^+:=(0,+\infty)$ for $\e_k>0$ such that $\Im_k(0,0)=1$, $\Im(w_1,w_2)((u_k,v_k)-(w_1,w_2))\in N$ for all $(w_1,w_2)\in B((0,0),\e_k)$. Choose $0<\rho <\e_k$ and $(h_1,h_2)\in Y$ such that $\|(h_1,h_2)\|=1$. Let $(w_1,w_2)_\rho:= \rho (h_1,h_2)$ then $\|(w_1,w_2)_\rho\|=\rho <\e_k$ and $(\theta_1,\theta_2)_\rho:=\Im_k((w_1,w_2)_\rho)((u_k,v_k)-(w_1,w_2)_\rho)\in N$ for each $k$. By Taylor's expansion and \eqref{PS-seq1}, since $(\theta_1,\theta_2)_\rho \in \mc N$ we get
\begin{equation}\label{PS-seq5}
\begin{split}
&\frac{1}{k}\|(u_k,v_k)- (\theta_1,\theta_2)_\rho\|
 \geq J(u_k,v_k)-J((\theta_1,\theta_2)_\rho)\\
 &= (J^\prime ((\theta_1,\theta_2)_\rho), (u_k-v_k)- (\theta_1,\theta_2)_\rho)+ o(\|(u_k,v_k)-(\theta_1,\theta_2)_\rho\|)\\
& = (1- \Im_k((w_1,w_2)_\rho)) (J^\prime ((\theta_1,\theta_2)_\rho), (u_k-v_k))+\rho\Im_k((w_1,w_2)_\rho) (J^\prime ((\theta_1,\theta_2)_\rho), (h_1,h_2)).
\end{split}
\end{equation}
We observe that
\[\lim_{\rho\to 0}\frac{1}{\rho}\|(\theta_1,\theta_2)_\rho-(u_k,v_k)\|= \|(u_k,v_k)(\Im_k^\prime(0,0),(h_1,h_2))- (h_1,h_2)\|.\]
Dividing \eqref{PS-seq5} by $\rho$ and passing to the limit as $\rho \to 0$ we derive
\[(J^\prime(u_k,v_k), (h_1,h_2))\leq \frac{1}{k}(\|(u_k,v_k)\|\|\Im_k^\prime(0,0)\|_{Y^*}+1).\]
From \eqref{IFT1} and \eqref{PS-seq5} we say that there exist a constant $C_2>0$ such that
\[\|\Im_k^\prime(0,0)\|_{Y^*} \leq \frac{C_2}{(p-1)\|u_k,v_k\|^p-(2q-1)L(u_k,v_k)}.\]
It remains to show that $(p-1)\|u_k,v_k\|^p-(2q-1)L(u_k,v_k) = (I^\prime (u_k,v_k),(u_k,v_k)) $ is bounded away from zero.  If possible let, for a subsequence, $\left| (I^\prime (u_k,v_k),(u_k,v_k))\right|= o(1)$ which implies \
\begin{equation}\label{PS-seq7}
\begin{split}
 (p-1) \|(u_k,v_k)\|^p  - (2q-1)L(u_k,v_k)= o(1)\\
\text{and}\;  (2q-p)\|(u_k,v_k)\|^p - (2q-1)K(u_k,v_k)=o(1).
\end{split}
\end{equation}
From \eqref{PS-seq5} and \eqref{PS-seq7}, it follows that there exist a constant $d_3>0$ such that
\[L(u_k,v_k)\geq d_3, \text{ for each } k.\]
Since $(u_k,v_k)\in \mc N$ we have
\[(p-1)K(u_k,v_k)- (2q-p)L(u_k,v_k)=o(1)\]
and \eqref{PS-seq7} gives
\[\left( \left(\frac{p-1}{2q-1} \right)\|(u_k,v_k)\|^p\right)^{\frac{2q-1}{2q-p}}- L(u_k,v_k)^{\frac{2q-1}{2q-p}}=o(1).\]
Altogether using all these along with Corollary \ref{inf-achvd-cor} we obtain
\begin{equation*}
\begin{split}
0 < \delta d_3^{\frac{p}{2q-p}}&\leq L(u_k,v_k)^{\frac{p-1}{2q-p}} G(u_k,v_k)\\
 &\leq C_{p,q} (\|(u_k,v_k)\|^p)\left({\frac{2q-1}{2q-p}}\right)- K(u_k,v_k) L(u_k,v_k)^{\frac{p-1}{2q-p}}\\
 & \leq \left(\frac{2q-p}{p-1}\right) \left( \left(\frac{p-1}{2q-1} \right)\|(u_k,v_k)\|^p\right)^{\frac{2q-1}{2q-p}}- L(u_k,v_k)^{\frac{2q-1}{2q-p}} =o(1)
\end{split}
\end{equation*}
which is a contradiction. This proves the claim.
Therefore we conclude that
\[\|J^\prime(u_k,v_k)\|_{Y^*} \to 0 \; \text{as} \;k \to 0\]
which proves our Lemma. \hfill{\QED}
\end{proof}

\begin{Lemma}\label{IFTN-}
Let $0\neq f_1,f_2 \in L^{\frac{q}{q-1}}(\mb R^n)$ satisfies \eqref{star0}. Given $(u,v)\in \mc N^-\setminus \{(0,0)\}$ ,there exist $\e>0$ and a differentiable function $\Im^- : B((0,0),\e) \subset Y \to \mb R^+:=(0,+\infty)$ such that $\Im^-(0,0)=1$, $\Im^-(w_1,w_2)((u,v)-(w_1,w_2))\in \mc N^-$ and
\begin{equation}\label{IFT1}
((\Im^-)^\prime(0,0),(w_1,w_2))= \frac{p(A(w_1,w_2)+ A_2(w_1,w_2))- q R(w_1,w_2)-\int_{\mb R^n} (f_1w_1+f_2w_2)}{(p-1)\|(u,v)\|^p - (2q-1)L(u,v)},
\end{equation}
for all $(w_1,w_2)\in B((0,0),\e)$ where $A_1(w_1,w_2):=\langle u,w_1\rangle+ \int_{\mb R^n}a_1(x)|u|^{p-2}uw_1$, $A_2(w_1,w_2):= \langle v,w_2\rangle+ \int_{\mb R^n}a_2(x)|v|^{p-2}vw_2 $ and
\begin{align*}
R(w_1,w_2) := &2\alpha \int_{\mb R^n} (|x|^{-\mu}* |u|^q)|u|^{q-2}uw_1 + 2\gamma \int_{\mb R^n} (|x|^{-\mu}* |v|^q)|v|^{q-2}vw_2 \\
& \quad + \beta \int_{\mb R^n} (|x|^{-\mu}* |u|^q)|v|^{q-2}vw_2+ \beta \int_{\mb R^n} (|x|^{-\mu}* |v|^q)|u|^{q-2}uw_1.
\end{align*}
\end{Lemma}
\begin{proof}
Fix $(u,v)\in \mc N^-\setminus \{(0,0)\}$, then obviously $(u,v)\in \mc N\setminus \{(0,0)\}$. Now arguing similarly as Lemma \ref{IFT}, we obtain the existence of $\e>0$ and a differentiable function $\Im^- : B((0,0),\e) \subset Y \to \mb R^+:=(0,+\infty)$ such that $\Im^-(0,0)=1$, $\Im^-(w_1,w_2)((u,v)-(w_1,w_2))\in\mc N$. Because $(u,v)\in \mc N^-$ we have
\[(I^\prime (u,v),(u,v))=(2q-p)\|u,v\|^p-(2q-1)\int_{\mb R^n}(f_1u+f_2v)~dx <0.\]
Since $I^\prime$ and $\Im^-$ are both continuous, they will not change sign in a sufficiently small neighborhood. So if we take $\e>0$ small enough then
\begin{align*}
&(I^\prime (\Im^-(w_1,w_2)((u,v)-(w_1,w_2))),(\Im^-(w_1,w_2)((u,v)-(w_1,w_2))))\\
&=(2q-p)\|\Im^-(w_1,w_2)((u,v)-(w_1,w_2))\|^p\\
&\quad \quad -(2q-1) \Im^-(w_1,w_2)\int_{\mb R^n}(f_1(u-w_1)+f_2(v-w_2))~dx <0
\end{align*}
which proves the Lemma.\hfill{\QED}
\end{proof}

\begin{Proposition}\label{PS-seq-N-}
Let  $0\neq f_1, f_2 \in L^{\frac{q}{q-1}}(\mb R^n)$ such that \eqref{star0}. Then there exists a sequence $(\hat u_m,\hat v_m) \subset \mc N^-$ such that $J(\hat u_m,\hat v_m) \to \Upsilon^- $ and $\|J^\prime(\hat u_m,\hat v_m)\|_{Y^*} \to 0$ as $m \to \infty$. .
\end{Proposition}
\begin{proof}
We note that $\mc N^-$ is closed, by Lemma \ref{N-closed}. Thus by Ekeland's variational principle we obtain a sequence $\{(\hat u_m,\hat v_m)\}$ in $\mc N^-$ such that
\begin{equation*}\label{PS-seq-N-1}
\left\{
\begin{split}
J(\hat u_m,\hat v_m) &\leq \Upsilon^- +\frac{1}{k},\; \text{and}\;\\
J(u,v) & \geq J(\hat u_m,\hat v_m)- \frac{1}{k}\|(\hat u_m-u,\hat v_m-v)\|, \; \text{for all}\; (u,v)\in \mc N^-.
\end{split}
\right.
\end{equation*}
By coercivity of $J$, $\{\hat u_m,\hat v_m\}$ forms a bounded sequence in $Y$. Then using Lemma \ref{IFTN-} and following the proof of Proposition \ref{PS-seq} we conclude the result.\hfill{\QED}
\end{proof}

Our next result shows that $J$ satisfies the $(PS)_c$ condition i.e. the Palais Smale condition for any $c\in \mb R$.

\begin{Lemma}\label{compact-PS-seq}
Let  $0\neq f_1, f_2 \in L^{\frac{q}{q-1}}(\mb R^n)$ such that \eqref{star0}. Then $J$ satisfies the $(PS)_c$ condition. That is, if $\{(u_k,v_k)\}$ is a sequence in $Y$ satisfying
\begin{equation}\label{PS-comp1}
J(u_k,v_k) \to c \; \text{and}\; J^\prime(u_k,v_k) \to 0, \; \text{as}\; k\to \infty,
\end{equation}
for some $c \in  \mb R$ then $\{(u_k,v_k)\}$ has a convergent subsequence.
\end{Lemma}
\begin{proof}
Let $\{(u_k,v_k)\}$ be a sequence in $Y$ satisfying \eqref{PS-comp1}. Using the same arguments as in Lemma \ref{PS-seq}(see \eqref{PS-seq4}), we can show that $\{(u_k,v_k)\}$ is bounded. Since $Y$ is reflexive Banach space, there exists a $(u,v)\in Y$ such that, upto a subsequence, $\{(u_k,v_k)\} \rightharpoonup (u,v)$ weakly in $Y$ as $k \to \infty$. Using compact embedding of $Y$ in $L^r(\mb R^n)$, for $r\in [p,p^*_s)$ i.e. Lemma \ref{conc-comp} we get $\{(u_k,v_k)\} \to (u,v)$ strongly in $L^r(\mb R^n)$  for $r\in (p,p^*_s)$  as $k \to \infty$. From weak continuity of $J^\prime $ and \eqref{PS-comp1} we get $J^\prime(u,v)=0$ .

 We claim that $\{(u_k,v_k)\} \to (u,v)$ strongly in $Y$. Since $\lim\limits_{k\to \infty}J^\prime(u_k,v_k)=0$, we consider
 \begin{equation}\label{PS-comp2}
 \begin{split}
 o_k(1)&= \langle u_k, (u_k-u)\rangle+ \langle v_k, (v_k-v)\rangle + \int_{\mb R^n}(a_1u_k(u_k-u)+ a_2v_k(v_k-v))\\
 & \quad -\left( \alpha \int_{\mb R^n}\int_{\mb R^n}\frac{|u_k(x)|^{q-2}u_k(x)(u_k-u)(x)|u_k(y)|^q}{|x-y|^{\mu}}~dxdy\right.\\
 &\quad + \gamma \int_{\mb R^n}\int_{\mb R^n}\frac{|v_k(x)|^{q-2}v_k(x)(v_k-v)(x)|v_k(y)|^q}{|x-y|^{\mu}}~dxdy\\
 & \quad + \beta \int_{\mb R^n}\int_{\mb R^n}\frac{|v_k(x)|^{q-2}v_k(x)(v_k-v)(x)|u_k(y)|^q}{|x-y|^{\mu}}~dxdy\\
 & \quad\left. + \beta \int_{\mb R^n}\int_{\mb R^n}\frac{|u_k(x)|^{q-2}u_k(x)(u_k-u)(x)|v_k(y)|^q}{|x-y|^{\mu}}~dxdy\right)\\
 & \quad - \int_{\mb R^n}(f_1(u_k-u)+f_2(v_k-v))~dx.
 \end{split}
 \end{equation}
 Since $q \in (q_l,q_u)$, $p< \frac{2nq}{2n-\mu}<p^*_s$. So using Proposition \ref{HLS}, we get
 \begin{equation}\label{PS-comp3}
 \begin{split}
 &\int_{\mb R^n}\int_{\mb R^n}\frac{|u_k(x)|^{q-2}u_k(x)(u_k-u)(x)|u_k(y)|^q}{|x-y|^{\mu}}~dxdy\\
 & \leq \left( \int_{\mb R^n} (|u_k|^{q-1}|u_k-u|)^{\frac{2n}{2n-\mu}} \right)^{\frac{2n-\mu}{2n}}\left( \int_{\mb R^n}|u_k|^{\frac{2nq}{2n-\mu}} \right)^{\frac{2n-\mu}{2n}}\\
 & \leq \left[ \left(\int_{\mb R^n}|u_k|^{\frac{2nq}{2n-\mu}}\right)^{\frac{q-1}{q}} \left(\int_{\mb R^n}|(u_k-u)|^{\frac{2nq}{2n-\mu}}\right)^{\frac{1}{q}}\right]^{\frac{2n}{2n-\mu}}\left( \int_{\mb R^n}|u_k|^{\frac{2nq}{2n-\mu}} \right)^{\frac{2n-\mu}{2n}}\\
 & = \left(\int_{\mb R^n}|(u_k-u)|^{\frac{2nq}{2n-\mu}}\right)^{\frac{(2n-\mu)}{2nq}}\left( \int_{\mb R^n}|u_k|^{\frac{2nq}{2n-\mu}} \right)^{\frac{(2n-\mu)(2q-1)}{2nq}} \to 0 \;\text{as}\; k \to \infty.
 \end{split}
 \end{equation}
 Similarly, we can get
\begin{equation}\label{PS-comp4}
\int_{\mb R^n}\int_{\mb R^n}\frac{v_k(x)|^{p-2}v_k(x)(v_k-v)(x)|v_k(y)|^p}{|x-y|^{\mu}}~dxdy \to 0  \;\text{as}\; k \to \infty
\end{equation}
and
\begin{equation}\label{PS-comp5}
\begin{split}
&\int_{\mb R^n}\int_{\mb R^n}\frac{|v_k(x)|^{q-2}v_k(x)(v_k-v)(x)|u_k(y)|^q}{|x-y|^{\mu}}~dxdy\\
& \quad +  \int_{\mb R^n}\int_{\mb R^n}\frac{|u_k(x)|^{q-2}u_k(x)(u_k-u)(x)|v_k(y)|^q}{|x-y|^{\mu}}~dxdy \to 0  \;\text{as}\; k \to \infty.
\end{split}
\end{equation}
Using the hypothesis on $f_1,f_2$ and H\"{o}lder's inequality, we can get
\begin{equation}\label{PS-comp6}
\int_{\mb R^n}(f_1(u_k-u)+f_2(v_k-v))~dx \to 0  \;\text{as}\; k \to \infty.
\end{equation}
Combining \eqref{PS-comp2}, \eqref{PS-comp3}, \eqref{PS-comp4}, \eqref{PS-comp5} and \eqref{PS-comp6}, we get
\begin{equation}\label{PS-comp7}
 o_k(1)= \langle u_k, (u_k-u)\rangle+ \langle v_k, (v_k-v)\rangle + \int_{\mb R^n}(a_1u_k(u_k-u)+ a_2v_k(v_k-v)).
\end{equation}
On a similar note, since $J^\prime(u,v)=0$, we get
 \begin{equation*}\label{PS-comp8}
 \begin{split}
 o_k(1)&= \langle u, (u_k-u)\rangle+ \langle v, (v_k-v)\rangle + \int_{\mb R^n}(a_1u(u_k-u)+ a_2v(v_k-v))\\
 & \quad -\left( \alpha \int_{\mb R^n}\int_{\mb R^n}\frac{|u(x)|^{q-2}u(x)(u_k-u)(x)|u(y)|^q}{|x-y|^{\mu}}~dxdy\right.\\
 &\quad + \gamma \int_{\mb R^n}\int_{\mb R^n}\frac{|v(x)|^{q-2}v(x)(v_k-v)(x)|v(y)|^q}{|x-y|^{\mu}}~dxdy\\
 & \quad + \beta \int_{\mb R^n}\int_{\mb R^n}\frac{|v(x)|^{q-2}v(x)(v_k-v)(x)|u(y)|^q}{|x-y|^{\mu}}~dxdy\\
 & \quad\left. + \beta \int_{\mb R^n}\int_{\mb R^n}\frac{|u(x)|^{q-2}u(x)(u_k-u)(x)|v(y)|^q}{|x-y|^{\mu}}~dxdy\right)\\
 & \quad - \int_{\mb R^n}(f_1(u_k-u)+f_2(v_k-v))~dx.
 \end{split}
 \end{equation*}
 Also, proceeding similarly as in \eqref{PS-comp3}-\eqref{PS-comp6}, we get
 \begin{equation}\label{PS-comp9}
 o_k(1)= \langle u, (u_k-u)\rangle+ \langle v, (v_k-v)\rangle + \int_{\mb R^n}(a_1u_k(u_k-u)+ a_2v_k(v_k-v)).
\end{equation}
Finally, \eqref{PS-comp7} and \eqref{PS-comp9} implies that
\[\lim_{k \to \infty} \|(u_k,v_k)-(u,v)\|^2 = 0\]
which proves our claim and consequently ends the proof.\hfill{\QED}
\end{proof}

\section{Existence of minimizers in $\mathcal{N^+}$ and $\mathcal{N^-}$}
In this section, we show that the minimums are achieved for $\Upsilon$ and $\Upsilon^-$.
\begin{Theorem}\label{min-achvd-N}
Let  $0\neq f_1, f_2 \in L^{\frac{q}{q-1}}(\mb R^n)$ such that \eqref{star0}. Then $\Upsilon$ is achieved at a point $(u_0,v_0)\in \mc N$ which is weak solution for $(P)$.
\end{Theorem}
\begin{proof}
From Proposition \ref{PS-seq}, we know that there exists a sequence $(u_k,v_k) \subset \mc N$ such that $J(u_k,v_k) \to \Upsilon $ and $\|J^\prime(u_k,v_k)\|_{Y^*} \to 0$ as $k \to \infty$.  Let $(u_0,v_0)$ be the weak limit of the sequence $\{(u_k,v_k)\}$ in $Y$. Since $(u_k,v_k)$ satisfies \eqref{PS-seq2} we get
\begin{equation}\label{maN}
\int_{\mb R^n}(f_1u_0+f_2v_0)~dx>0.
\end{equation}
Also $\|J^\prime(u_k,v_k)\|_{Y^*} \to 0$ as $k \to \infty$ implies that
\[(J^\prime(u_0,v_0),(\phi_1,\phi_2))=0,\; \text{for all}\; (\phi_1,\phi_2)\in Y\]
i.e. $(u_0,v_0)$ is a weak solution of $(P)$. In particular $(u_0,v_0)\in \mc N$. Moreover
\[\Upsilon \leq J(u_0,v_0) \leq \liminf_{k \to \infty} J(u_k,v_k)= \Upsilon\]
which implies that $(u_0,v_0)$ is the minimizer for $J$ over $\mc N$. \hfill{\QED}
\end{proof}

\begin{Corollary}\label{min-achvd-N+}
Let $(u_0,v_0)\in\mc N$ be such that $\Upsilon= J(u_0,v_0)$ then $(u_0,v_0)\in \mc N^+$
and $(u_0,v_0)$ is a local minimum for $J$ in $Y$.
\end{Corollary}
\begin{proof}
 Since \eqref{maN} holds, using Lemma \ref{N0-empty} we get that there exists $t_1, t_2>0$ such that $(u_1,v_1):=(t_1u_0, t_1v_0) \in \mc N^+$ and $(t_2u_0, t_2v_0) \in \mc N^-$. We claim that $t_1=1$ i.e. $(u_0,v_0)\in \mc N^+$. If $t_1<1$ then $t_2=1$ which implies $(u_0,v_0)\in \mc N^-$. Now $J(t_1u_0, t_1v_0)\leq J(u_0,v_0)= \Upsilon $ which is a contradiction to $(t_1u_0, t_1v_0) \in \mc N^+$. To show that $(u_0,v_0)$ is also a local minimum for $J$ in $Y$, we first notice that for each $(u,v)\in Y$ with $K(u,v)>0$ we have
 \[J(\hat tu,\hat tv)\geq J(\hat t_2u,\hat t_2v) \; \text{whenever}\; 0<\hat t< \left(\frac{(p-1)\|(u,v)\|^p}{(2q-1)L(u,v)}\right)^{\frac{1}{2q-p}},\]
where $\hat t_2$ is corresponding to $(u,v)$. In particular, if $(u,v)= (u_0,v_0)\in \mc N^+$ then
\[t_2 =1 < t_0 = \left(\frac{(p-1)\|(u_0,v_0)\|^p}{(2q-1)L(u_0,v_0)}\right)^{\frac{1}{2q-p}}.\]
Using Lemma \ref{IFT}, we obtain a differentiable map $\Im : B((0,0),\e) \to \mb R^+$ for $\e>0$ such that $\Im(w_1,w_2)((u_0,v_0)- (w_1,w_2)) \in \mc N$ whenever $\|(w_1,w_2)\|< \e$. Choosing $\e>0$ sufficiently small so that
\begin{equation}\label{min-achvd-N+1}
1 < \left(\frac{(p-1)\|((u_0,v_0)- (w_1,w_2))\|^p}{(2q-1)L((u_0,v_0)- (w_1,w_2))}\right)^{\frac{1}{2q-p}}
\end{equation}
for every $(w_1,w_2)\in B((0,0),\e)$. This implies $\Im(w_1,w_2)((u_0,v_0)- (w_1,w_2)) \in \mc N^+$ and whenever $0<\hat t< \displaystyle\left(\frac{(p-1)\|((u_0,v_0)- (w_1,w_2))\|^p}{(2q-1)L((u_0,v_0)- (w_1,w_2))}\right)^{\frac{1}{2q-p}}$ we have
\[J(\hat t((u_0,v_0)- (w_1,w_2))) \geq J(\Im(w_1,w_2)((u_0,v_0)- (w_1,w_2))) \geq J((u_0,v_0)).\]
Since \eqref{min-achvd-N+1} holds, we can take $\hat t=1$ and this gives
\[J(\hat t((u_0,v_0)- (w_1,w_2))) \geq  J(u_0,v_0)\; \text{whenever}\; \|(w_1,w_2)\|<\e\]
which prove the last assertion.\hfill{\QED}
\end{proof}

\noi\textbf{Proof of Theorem \ref{mainthrm}:} The proof follows from Theorem \ref{min-achvd-N} and Corollary \ref{min-achvd-N+} except that we need to show that there exist a non negative solution if $f_1,f_2 \geq 0$. Suppose $f_1,f_2 \geq 0$ then consider the function $(|u_0|, |v_0|)$. Then we know that there exist a $t_1>0$ such that $(t_1|u_0|, t_1|v_0|) \in \mc N^+$ and $t_1|u_0|, t_1|v_0|>0$. It is easy to see that $\|(u_0,v_0)\|= \|(|u_0|,|v_0|)\|$, $L(u_0,v_0)= L(|u_0|,|v_0|)$ and $K(u_0,v_0) \leq K(|u_0|,|v_0|)$. If $\varphi_{u,v}(t)$ denotes the fibering map corresponding to $(u,v) \in Y$ as introduced in section 3, we get
$\varphi^\prime_{|u_0|,|v_0|}(1) \leq \varphi^\prime_{u_0,v_0}(1)=0$. Since $t_1$ is the point of local minimum of $\varphi_{|u_0|,|v_0|}(t)$ for $0<t<\displaystyle \left(\frac{(p-1)\|(|u_0|,|v_0|)\|^p}{(2q-1)L(|u_0|,|v_0|)}\right)^{\frac{1}{2q-p}}$, $t_1 \geq 1$. Necessarily,
\[ J(t_1|u_0|,t_1|v_0|)\leq J(|u_0|,|v_0|)\leq J(u_0,v_0)\]
which implies that we can always take $u_0, v_0 \geq 0$ while considering the weak solution $(u_0,v_0)$ for $(P)$. \hfill{\QED}

\noi Next we prove that the infimum $\Upsilon^-$ is achieved and the minimizer  is another weak solution to problem $(P)$.
\begin{Theorem}\label{min-achieved2}
Let  $0\neq f_1, f_2 \in L^{\frac{q}{q-1}}(\mb R^n)$ such that \eqref{star0}, then there exists $(u_1,v_1)\in \mc N^-$ such that
\[ \Upsilon^- = J(u_1,v_1). \]
\end{Theorem}
\begin{proof}
Using Lemma \ref{PS-seq-N-} we know that there exist a sequence $\{(\hat u_m,\hat v_m)\} \subset \mc N^-$ such that
\[J(\hat u_m, \hat v_m)\to \Upsilon^-\; \text{and}\; J^\prime(\hat u_m,\hat v_m) \to 0, \; \text{as}\; m \to \infty.\]
 Applying again Lemma \ref{compact-PS-seq}, we get that there exists $(u_1,v_1) \in Y$ such that, upto a subsequence, $(\hat u_m,\hat v_m) \to (u_1,v_1)$ strongly in $Y$ as $m \to \infty$. Since $L(u,v)$ is weakly sequentially continuous, this implies $\lim\limits_{k \to \infty}J(\hat u_m,\hat v_m)= J(u_1,v_1)=\Upsilon^-$ and $(u_1,v_1) \in \mc N^-$.. Therefore, Lemma \ref{exist-sol} implies that $(u_1,v_1)$ is a weak solution of $(P)$. \hfill{\QED}\\
\end{proof}

\noi  Finally, we prove Theorem \ref{mainthrm2}.

\noi \textbf{Proof of Theorem \ref{mainthrm2}:} The existence of second weak solution $(u_1,v_1)$ for $(P)$ is asserted by Theorem \ref{min-achieved2}. So we only need to show that we can obtain a non negative weak solution if $f_1,f_2\geq 0$. Consider the function $(|u_1|,|v_1|)$ then there exist a $t_2>0$ such that $(t_2|u_1|, t_2|v_1|) \in \mc N^-$. Let
\[t_0= \left(\frac{(p-1)\|(u_1,v_1)\|^p}{(2q-1)L(u_1,v_1)}\right)^{\frac{1}{2q-p}}\]
then since $(u_1,v_1)\in \mc N^-$ we conclude that
\[J(u_1,v_1)= \max_{t\geq t_0} J(tu_1,tv_1)\geq J(t_2u_1,t_2v_1) \geq J(t_1|u_1|,t_1|v_1|).\]
Therefore it remains true to consider $u_1,v_1 \geq 0$ while considering the weak solution $(u_1,v_1)$ in case $f_1,f_2\geq0$.
 \linespread{0.5}

\end{document}